\documentclass[10pt]{article}
\usepackage{amsmath,amsthm,amsfonts,amssymb,amscd, amsxtra}
\theoremstyle{plain}
\newtheorem{thm}{Theorem}[section]

\newtheorem{defn}{Definition}[section]

\newtheorem{rmk}{Remark}[section]

\newtheorem{assumption}{Assumption}[section]
\newtheorem{algorithm}{Algorithm}[section]
\newcommand{\beq}{\begin{equation}}
\newcommand{\eeq}{\end{equation}}
\newcommand{\beqa}{\begin{eqnarray}}
\newcommand{\eeqa}{\end{eqnarray}}
\newcommand{\beqas}{\begin{eqnarray*}}
\newcommand{\eeqas}{\end{eqnarray*}}

\makeatletter
\renewcommand*{\@biblabel}[1]{\hfill#1.}
\makeatother

\begin{document}
\title{Generalized Inexact Proximal Algorithms: Habit's/ Routine's Formation
with Resistance to Change, following Worthwhile Changes}

\author{G. C. Bento\thanks{IME-Universidade Federal de Goi\'as,
Goi\^ania-GO 74001-970, Caixa Postal: 131, BR, \newline
Fone: +55 (62)3521-1208  ({\tt glaydstonc@gmail.com}) }
\and
A. Soubeyran \thanks{Aix-Marseille School of Economics, Aix-Marseille University, CNRS \& EHESS, FR ({\tt antoine.soubeyran@gmail.com})}
}

\maketitle
\vspace{.2cm}

%
%
%
%
\begin{abstract}
This paper shows how, in a quasi metric space, an inexact proximal algorithm
with a generalized perturbation term appears to be a nice tool for
Behavioral Sciences (Psychology, Economics, Management, Game theory,\ldots
). More precisely, the new perturbation term represents an index of
resistance to change, defined as a \textquotedblleft curved enough" function
of the quasi distance between two successive iterates. Using this behavioral
point of view, the present paper shows how such a generalized inexact
proximal algorithm can modelize the formation of habits and routines in a
striking way. This idea comes from a recent \textquotedblleft variational
rationality approach" of human behavior which links a lot of different
theories of stability (habits, routines, equilibrium, traps,\ldots ) and
changes (creations, innovations, learning and destructions,\ldots ) in
Behavioral Sciences and a lot of concepts and algorithms in Variational
Analysis. In this variational context, the perturbation term represents a
specific instance of the very general concept of resistance to change, which
is the disutility of some inconvenients to change. Central to the analysis
are the original variational concepts of \textquotedblleft worthwhile
changes" and \textquotedblleft marginal worthwhile stays". At the behavioral
level, this paper advocates that proximal algorithms are well suited to
modelize the emergence of habituation/routinized human behaviors. We show
when, and at which speed, a \textquotedblleft worthwhile to change" process
converges to a behavioral trap.

\noindent\textbf{AMS Classification.} 49J52\,$\cdot$\,49M37\,$\cdot$\,65K10\,%
$\cdot$\,90C30\,$\cdot$\,91E10.

\noindent {\bf keywords}.\  {Nonconvex optimization\and Kurdyka-Lojasiewicz inequality\and
inexact proximal algorithms\and habits\and routines\and worthwhile changes}
\end{abstract}

\section{Introduction}

\label{sec:1}

The main message of this paper is that, using the behavioral context of a
recent \textquotedblleft Variational rationality" approach of worthwhile
stay and change dynamics proposed by Soubeyran~\cite{Soubeyran2009,
Soubeyran2010}, a generalized proximal algorithm can modelize fairly well an
habituation process as described in Psychology for an agent, or a
routinization process, in Management Sciences, for an organization.
This opens the door to a new vision of proximal algorithms. They are not
only very nice mathematical tools in optimization theory, with striking
computational aspects. They can also be nice tools to modelize the dynamics
of human behaviors.

Theories of stability and change consider successions of stays and changes.
Stays refer to habits, routines, equilibrium, traps, rules, conventions,\ldots. Changes represent creations, destructions, learning processes,
innovations, attitudes as well as beliefs formation and revision, self
regulation problems, including goal setting, goal striving and goal
revision, the formation and break of habits and routines,\ldots. In the
interdisciplinary context which characterizes all these theories in
Behavioral Sciences, the \textquotedblleft Variational rationality
approach" (see \cite{Soubeyran2009, Soubeyran2010}), shows how to modelize
the course of human activities as a succession of worthwhile temporary stays
and changes which balance, each step, motivation to change (the utility of
advantages to change) and resistance to change (the disutility of
inconvenients to change). This very simple idea has allowed to see proximal
algorithms as an important tool to modelize the human course of actions,
where the perturbation term of a proximal algorithm can be seen as a crude
formulation of the complex concept of resistance to change, while the
utility generated by a change in the objective function can represent a
crude formulation of the motivation to change concept. The\ Variational
rationality approach considers three original concepts: \ i)
\textquotedblleft worthwhile changes", when, each step, motivation to change
is higher enough with respect to resistance to change, ii) \textquotedblleft
non marginal worthwhile changes" and, \ iii) ``variational traps",
\textquotedblleft easy enough to reach", that the agent can reach using a
succession of worthwhile changes, and \textquotedblleft difficult enough to
leave", such that, being there, it is not worthwhile to move from there.

These three concepts represent the pilar of the Variational rationality
approach; see \cite{Soubeyran2009, Soubeyran2010}, which has provided an
extra motivation to develop further the study of proximal algorithms in a
nonconvex and possibly nonsmooth setting. Among other recent applications of
this simple idea, see Attouch and Soubeyran~\cite{AtAnt2011} for local
search proximal algorithms, Flores-Bazan et al. \cite{Antoine2012} for
worthwhile to change games, Attouch et al. \cite{ARS2007} for alternating
inertial games with costs to move, and Cruz Neto et al. \cite{Xavier2012}
for the \textquotedblleft how to play Nash" problem,\ldots . In all these
papers the perturbation term of the usual proximal point algorithm is a
linear or a quadratic function of the distance or quasi distance between two
successive iterates. They modelize the case of \textquotedblleft strong
enough resistance" to change. Our paper examines the opposite case of
\textquotedblleft weak enough" resistance to change where the pertubation
term modelizes the difficulty (relative resistance) to be able to change as
a \textquotedblleft curved enough" function of the quasi distance between
two successive iterates. A quasi distance modelizes costs to be able to
change as an index of dissimilarity between actions where the cost to be
able to change from an action to an other one is not the same as the cost to
be able to change in the other way. In a first paper, Bento and Soubeyran~\cite{BSPaper1}
 show when, in a quasi metric space, a generalized inexact proximal
algorithm, equipped with such a generalized perturbation term, defined, each
step, by a sufficient descent condition and a stopping rule, converges to a
critical point. Then, it is shown that the speed of convergence and
convergence in finite time depend of the curvature of the perturbation term
and of the Kurdyka-Lojasiewicz property associated to the objective
function. A striking and new application is given. It concerns the impact of
the famous ``loss aversion effect" (Nobel Prize Kahneman and Tversky~\cite{Tversky1979},
Tversky and Kahneman~\cite{Tversky1991}) on the speed of convergence of the generalized
inexact proximal algorithm.

In the present paper, inspired by the VR \textquotedblleft variational
rationality" approach, we consider a new inexact proximal algorithm whose
sufficient descent condition is, each step, a little more demanding, using,
each step, the same stopping rule. Applying the convergent result of the
first paper \cite{BSPaper1} this simple modification is a way to
force convergence even more. It gives an intuitive sufficient condition for
the critical point to be a variational trap (weak or strong). In this case
changes are required to be " worthwhile enough", the stopping rule is the
same, and the end of the convergent worthwhile stay and change process is
both a critical point and a variational trap. Doing so, this paper extends
the convergence result to a critical point of Attouch and Bolte~\cite{Attouch2009}, Attouch et al.~\cite{Attouch2010} and Moreno et al.~\cite{Moreno2011}, using a fairly general \textquotedblleft convex enough"
perturbation term. It is important to note that, as an application, it is
possible to consider the formation of habits and routines as an inexact
proximal algorithm in the context of weak resistance to change. However,
because of its strongly interdisciplinary aspect (Mathematics, Psychology,
Economics, Management), to be carefully justified, this application needs
several steps. Due to space constraints, these considerations are given in
Bento and Antoine~\cite{BentoAntoine2014-2}. Then, at the behavioral level, the main message of this paper is
to advocate that our generalized proximal algorithm is well suited to
modelize the formation of habitual/ routinized human behaviors. The list of the main (VR)
concepts is presented on an example in Section 2. To get more perspective,
in the Annex, the VR approach of stability and change dynamics is compared
in great detail with a complementary theory relative to the dynamics of
human behavior, the HD habitual domain theory (see Yu~\cite{Yu1990} ) and its recent
continuation, the DMCS approach (decision making with changeable spaces; see Larbani and Yu~\cite{LarbaniYu2012}). A last section applies the convergence and speed of
convergence result of our generalized proximal algorithm to the modelization
of the formation and break of habits and routines. It compares very
succinctly what HD and VR can say on this important topic for human behavior.

At a higher dimensional level where inexact proximal \ algorithms represent
a specific formulation of VR, the complementary VR and HD
approaches. consider, both stability and change dynamics, but modelize them in
a different way: deterministic worthwhile temporary stays and changes
dynamics for the VR approach and Markov chains for HD theory. Both focus the
attention on optimization and satisficing processes. VR is variable and
possibly intransitive preference (utility) based, while HD is charge based.
VR main topic is the self regulation problem (goal setting, goal striving,
goal revision and goal disengagement) at the individual level or for
interactive agents. The HD main topic is to know ``how agents expand and
enrich their habitual domain". Both approaches examine decision making
problems with changeable structures (spaces, parameters, goals, preferences
and charges). Each approach considers different aspects of what can be
changed and how it can change.


Our paper is organized as follows. Section~\ref{sec2} gives an example which
helps to list the main variational tools necessary to define the central
concepts of \textquotedblleft worthwhile change" and "variational trap" for
behavioral applications. Section 3 shows how inexact proximal
algorithms can represent adaptive satisficing processes. Section~4
examines a generalized inexact proximal algorithm which converges to a
critical point \ which is also a variational trap (weak or strong), when the
objective function satisfies a Kurdyka-Lojasiewicz inequality. A last
section summarizes very briefly\ the VR variational rationality approach and
the HD habitual domain theory. The conclusion follows. In \cite{BentoAntoine2014-2} the authors
compare in greater details the VR and HD approaches relative to
habituation/routinization processes.

\section{Variational Rationality : How Successions of Worthwhile Stays and
Changes End in Variational Traps}\label{sec2}

\subsection{Worthwhile Stay and Change Dynamics}

A recent variational rationality approach, see \cite%
{Soubeyran2009, Soubeyran2010}, gives a common background to a lot of
theories of stability/stay and change in Behavioral Sciences (Psychology,
Economics, Management Sciences, Decision theory, Philosophy, Game theory,
Political Sciences, Artifical Intelligence$\ldots $), using as a central
building bloc the three concepts of \textquotedblleft worthwhile change",
\textquotedblleft marginal worthwhile change" and ``variational trap". All
these behavioral dynamics can be seen as a succession of worthwhile
temporary stays and changes $x^{k+1}\in W_{e_{k},\xi _{k+1}}(x^{k})$, $k\in \mathbb{N}$,
ending in variational traps $x^{\ast }\in X$, where $X$ is the universal space of actions (doing), having or being, depending of the applications.
$X$ includes all past elements and all the new elements that can be
discovered as time evolves.

The main idea is quite evident. If a behavioral theory wants to explain
\textquotedblleft why, where, how and when" agents perform actions and
change, this theory must define, each period, along a \quad path of changes $%
\left\{ x^{0},x^{1},...,x^{k},x^{k+1},....\right\} $ why the agent have,
first, an incentive to do some steps away from his current position and,
then, an incentive to stop changing one step more within this period.\ \ In
the current period $k+1,$ a change is such that $x^{k+1}\neq x^{k},$ while a
stay is $x^{k+1}=x^{k}$. Let $e_{k}\in E$ be the experience of the agent at
the end of the last period $k.$ A change $x^{k}$ $\curvearrowright x^{k+1}\in W_{e_{k},\xi _{k+1}}(x^{k})$
is worthwhile, when his ex ante motivation to change $M_{e_{k}}(x^{k},x^{k+1})$ is sufficiently higher (more than $\xi _{k+1}>0$)
\ than his ex ante resistance to change, $R_{e_{k}}(x^{k},x^{k+1})$. Then, $
x^{k+1}\in W_{e_{k},\xi _{k+1}}(x^{k})\Longleftrightarrow
M_{e_{k}}(x^{k},x^{k+1})\geq \xi _{k+1}R_{e_{k}}(x^{k},x^{k+1}).$ Motivation
and resistance to change are two complex variational concepts which admit a
lot of variants (see \cite{Soubeyran2009, Soubeyran2010}). Motivation to
change $M_{e_{k}}(x^{k},x^{k+1})=U_{e_{k}}\left[ A_{a_{k}}(x^{k},x^{k+1})%
\right] $ is the utility $U_{e_{k}}\left[ \cdot\right] $ of advantages to
change, $A_{e_{k}}(x^{k},x^{k+1}),$ while resistance to change $%
R_{e_{k}}(x^{k},x^{k+1})=D_{e_{k}}\left[ I_{e_{k}}(x^{k},x^{k+1})\right] $
is the disutility $D_{e_{k}}\left[ \cdot \right] $ of inconvenients to change $%
I_{e_{k}}(x^{k},x^{k+1}).$

\textbf{Worthwhile changes are generalized satisficing changes}: Within a
period, a worthwhile change $x^{k}\curvearrowright x^{k+1}\in W_{e_{k},\xi
_{k+1}}(x^{k})$ is desirable and feasible enough, i.e, acceptable, improving
with no too high costs to be able to improve. Then, a worthwhile change is a
generalized satisficing change where, each period, $\ $the agent chooses the
ratio $\xi _{k+1}>0,$ which \ represents how worthwhile a change must be to
accept to move rather than to stay. The famous Simon~\cite{simon1955} 
satisficing principle is a specific case (see \cite{Soubeyran2009,Soubeyran2010}).
Second, within the same period, the agent must also have to know when he
must stop changing. This is the case when one step more is not worthwhile.
More formally, this change is not \textquotedblleft marginally worthwhile",
when the ex ante marginal motivation to change is sufficiently lower than
the ex ante marginal resistance to change. In this case the agent does not
regret ex ante to do not go one step further. The motivation to change again
next period comes from residual unsatisfied needs or variable preferences.

A variational trap $x^{\ast }$ is such that, starting from an initial point $%
x^{0}\in X,$ it exists a path of worthwhile changes $x^{k+1}\in W_{e_{k},\xi
_{k+1}}(x^{k})$ which ends in $x^{\ast },$ i.e., such that, being there, it
is not worthwhile to move again, i.e., $W_{e_{\ast },\xi _{\ast }}(x^{\ast
})=\left\{ x^{\ast }\right\} $.

\subsection{Variational Concepts. An Example}

To save space and to fix ideas, let us define on a simple example all these
variational rationality concepts. This being done, we can easily show how an
inexact proximal algorithm represents a nice benchmark process of worthwhile
temporary stays and changes in term of, each \ period, \ a sufficient
descent condition and a stopping rule. For much more comments, and a more
complete formulation of each of these variational concepts, with references
to a \ lot of different disciplines in Behavioral Sciences which help to
justify their unifying power; see \cite{Soubeyran2009, Soubeyran2010}.

\begin{itemize}
\item [] \textbf{A simple model of knowledge management}: This example modelizes a
very simple case of knowledge management within an organization, to
determine a satisficing or, as a extreme case, the optimal size and shape of an innovative firm driven by a
leader. In Management Sciences, the literature on this topic is enormous and
represents one of its main area of research. Consider an entrepreneur
(leader) who, each period, can hire and fire different kinds and numbers of
skilled and specialized workers $\left\{ 1,2,...,j,\ldots,l\right\} =J$ (say
knowledge workers; see Long et al.~\cite{Long2014}) to produce a chosen
quantity of a final good of a chosen quality. The endogenous
quality $q(x)$ of this final good changes with the chosen profile of skilled
workers $x=(x^{1},x^{2},..,x^{j},..,x^{l})\geq 0$, where $x^{j}\geq 0$ is a
number of workers of type \ $j$. To save space and for simplification, each
period, each employed skilled worker utilizes one unit of a specific non
durable mean to produce, using his specific knows how to produce, one unit of a
specific component of type $j$. Then, the entrepreneur combines these
different components to produce $\mathfrak{q}(x)$ units of a final good of
endogenous quality $s(x)$. This production function is original because it
mixes both variable quantity and quality. The revenue of the entrepreneur
is $\varphi \left[ \mathfrak{q}(x),s(x)\right]$. His operational costs $%
\rho (x)$ are the sum of his costs to buy the non durable means used by
each worker, and the wages paid to each employed worker. Then, in a given
period, the profit of the entrepreneur who employs the profile $x\in X=R^{l}$
of skilled workers is $g(x)=\varphi \left[ \mathfrak{q}(x),s(x)\right] -\rho
(x)\in R.$ For a famous example of an endogenous production function of
quality, see Kremer~\cite{Kremer1993}.

\item [] \textbf{Advantages to change}: Let $x=x^{k}$ and $y=x^{k+1}$ be the last
period, and current period profiles of skilled workers chosen by the
entrepreneur. Then, if this is the case, his advantages to change his
profile of skilled workers from one period to the next is $%
A(x,y)=g(y)-g(x)\geq 0$;

\item [] \textbf{Inconvenients to change}: They represent the difference $%
I(x,y)=C(x,y)-C(x,x)\geq 0$ between costs $C(x,y)$ to be able to change from
profile $x$ to profile $y$ and costs $C(x,x)$ to be \ able to stay with the
same profile $x$ used in the last period;

\item [] \textbf{Costs to be able to change (to stay)}: To be able to hire one
skilled worker of type $j$, ready to work, costs \ $h_{+}^{j}>0$. These
costs include search and training costs. To fire one worker of type $j$,
costs $h_{-}^{j}>0$. These costs represent separation and compensation
costs. To keep \ a worker, ready to work, one period more, costs $%
h_{=}^{j}\geq 0$. These conservation costs include knowledge regeneration
and motivation costs. Then, in the current period, i) costs to conserve the
same profile of workers as in the last period are $C(x,x)=\Sigma
_{j=1}^{n}h_{=}^{j}x^{j}$ while , ii) costs to utilize the profile of
skilled workers $y$ are:
\[
C(x,y)=\Sigma _{j\in J_{+}(x,y)}\left[ h_{=}^{j}x^{j}+h_{+}^{j}(y^{j}-x^{j})%
\right] +\Sigma _{j\in J_{-}(x,y)}\left[
h_{=}^{j}y^{j}+h_{+}^{j}(x^{j}-y^{j})\right] ,
\]
where $J_{+}(x,y)=\left\{ j\in J,y^{j}\geq x^{j}\right\} $ and $%
J_{-}(x,y)=\left\{ j\in J,y^{j}<x^{j}\right\}$. For simplification, suppose that conservation costs are zero, i.e., $%
h_{=}^{j}=0$. Then, 
\[
C(x,x)=0\quad \mbox{and}\quad I(x,y)=\Sigma _{j\in
J_{+}(x,y)}h_{+}^{j}(y^{j}-x^{j})+\Sigma _{j\in
J_{-}(x,y)}h_{+}^{j}(x^{j}-y^{j}).
\]
So, $I(x,y)$ is a quasi distance $q(x,y):=I(x,y)\geq 0$ such that
\begin{itemize}
\item [i)] $q(x,y)=0$ iff $y=x$;
\item [ii)] $q(x,z)\leq q(x,y)+q(y,z)$, $x,y,z\in X$.
\end{itemize}
The more general case where $h_{=}^{j}>0 $ works as well.

\item [] \textbf{Motivation and resistance to change functions}: They are, moving
from the past profile of knowledge workers $x$ to the current profile $y$
are 
\[
M(x,y)=U\left[ A(x,y)\right] =\left[ g(y)-g(x)\right] ^{\mu }\quad \mbox{and}%
\quad R(x,y)=D\left[ I(x,y)\right] =q(x,y)^{\nu },\quad \mu ,\nu >0, 
\]%
where the utility and disutility functions are $U\left[ A\right] =A^{\mu }$
and $D\left[ I\right] =I^{\nu }$.

\item [] \textbf{Relative resistance to change function}: It is $\Gamma \left[
q(x,y)\right] =U^{-1}\left[ D\left[ I(x,y)\right] \right] =q(x,y)^{\nu /\mu
} $, where $\alpha =\nu /\mu >0.$

\item [] \textbf{Worthwhile changes}: In this setting, a change from
profile $x$ to profile $y$ is worthwhile if $M(x,y)\geq \xi R(x,y)$, i.e., $%
\left[ g(y)-g(x)\right]^{\mu }\geq \xi q(x,y)^{\nu }$, where $\xi >0$ is
the current and chosen \textquotedblleft worthwhile enough" satisficing
ratio. Then, a worthwhile change is such that 
\[
y\in W_{\xi }(x)\Longleftrightarrow g(y)-g(x)\geq \lambda \Gamma \left[
q(x,y)\right] ,\qquad \lambda =(\xi )^{1/\mu }>0. 
\]

\item [] \textbf{Succession of worthwhile temporary stays and changes}: In this
example they are
\[
g(x^{k+1})-g(x^{k})\geq \lambda _{k+1}\Gamma \left[ q(x^{k},x^{k+1})\right], \qquad k\in \mathbb{N}.
\]
\item [] \textbf{Variational traps: }In the example, given the initial profile of
skilled workers $x^{0}\in X,$ and a final worthwhile enough to change ratio $%
\lambda _{\ast }>0,$ $x^{\ast }\mathbf{\in }X$ \ is a variational trap if it
exists a path of worthwhile temporary stays and changes {\huge \ }$\left\{
x^{0},x^{1},...,x^{k},x^{k+1},....\right\} $ such that,
\begin{itemize}
\item [i)] $g(x^{k+1})-g(x^{k})\geq \lambda _{k+1}\Gamma \left[ q(x^{k},x^{k+1})%
\right] $, $k\in \mathbb{N}$;
\item [ii)] $g(y)-g(x^{\ast })<\lambda _{\ast }\Gamma \left[
q(x^{k},y)\right] ,$ $y\neq x^{*}, y\in X$.
\end{itemize}

\item [] \textbf{An habituation/routinization process:} It is such that, step by
step, gradually, the agent carries out a more and more similar action. This
is equivalent to say than the quasi distance $C(x^{k},x^{k+1})$ converges to
zero as $k$ goes to infinite.
\end{itemize}
When a worthwhile to change process converges to a variational trap, this
variational formulation offers a model of trap as the end point of a path of
worthwhile changes.

\section{Inexact Proximal Algorithms as Worthwhile Stays and Changes
Processes}

\subsection{Inexact Proximal Formulation of Worthwhile Changes}
\begin{itemize}
\item [] \textbf{Proximal intransitive preferences}. Let us define, in the current
period $k+1,$ the ``to be increased" entrepreneur proximal payoff to change
from $x=x^{k}$ to $y=x^{k+1}$ as $Q_{\lambda }(x,y)=g(y)-\lambda \Gamma %
\left[ q(x,y)\right] $ with $\lambda >0$. Then, the proximal payoff to stay
at $x=x^{k}=y=x^{k+1}$ is $Q_{\lambda }(x,x)=g(x)-\lambda \Gamma \left[
q(x,x)\right] =g(x).$ It follows that it is\ worthwhile\ to\ change from profile 
$x$ to profile $y$ iff $Q_{\lambda }(x,y)\geq Q_{\lambda }(x,x)$, i.e., $y\in
W_{\lambda }(x)$. This defines a variable and possibly non transitive
preference $z\geq _{x,\lambda }y\Longleftrightarrow Q_{\lambda }(x,z)\geq
Q_{\lambda }(x,y).$

To fit with the formulation of inexact proximal algorithms, where
mathematicians consider ``to be decreased" cost functions, let us consider
the residual profit that the entrepreneur expects to exhaust in the  future, 
$f(x)=\overline{g}-g(x)\geq 0$, where $\overline{g}$ $=\sup \left\{
g(y),y\in X\right\} <+\infty $ is the highest finite profit that the
entrepreneur can hope to get. Then, the ``to be decreased" proximal payoff \
of the entrepreneur is
\begin{equation}\label{eq:10000}
P_{\lambda }(x,y)=f(y)+\lambda \Gamma \left[ q(x,y)%
\right].
\end{equation} 
In this case, to move from profile $x$ to profile $y$ is a
worthwhile change $y\in W_{\lambda }(x)$ iff $P_{\lambda }(x,y)\leq
P_{\lambda }(x,x)$.

\item [] \textbf{Sufficient descent methods. }The entrepreneur performs, each period
\ $k+1,$ a sufficient descent, if he can choose a new profile $x^{k+1}$ such
that $f(x^{k})-f(x^{k+1})\geq {\lambda _{k+1}}\Gamma \lbrack
q(x^{k},x^{k+1})]$. This means that the entrepreneur follows a path of
worthwhile changes $x^{k+1}\in W_{\lambda _{k+1}}(x^{k})$, $k\in \mathbb{N}$. Since $q(x^{k},x^{k})=0$, this comes from definition of $W_{\lambda _{k+1}}(x^{k})$ combined with \eqref{eq:10000} for $x=x^{k}$, $y=x^{k+1}$ and $\lambda=\lambda_{k+1}$.
%
In this case, each worthwhile change is not optimizing, contrary to each
step of an\ exact proximal algorithm.

\item [] \textbf{Exact proximal algorithms.} The entrepreneur follows an exact
proximal algorithm if, each current period \ $k+1$, he can choose a new
profile $x^{k+1}$ which minimizes his ``to be decreased" proximal payoff $%
P_{\lambda _{k+1}}(x^{k},y)=f(y)+\lambda _{k+1}\Gamma \lbrack q(x^{k},y)]$ on
the whole space $X$, \textbf{\ }%
\begin{equation}
x^{k+1}\in \mbox{argmin}_{y\in X}\left\{ f(y)+\lambda
_{k+1}\Gamma \lbrack q(x^{k},y)]\right\},\qquad k\in \mathbb{N},
\end{equation}
which allows us to obtain  $x^{k+1}\in W_{\lambda _{k+1}}(x^{k})$, $k\in \mathbb{N}$. In
Mathematics the formulation is
\begin{equation}\label{eq:prox1-2}
x^{k+1}\in \mbox{argmin}_{y\in X}\left\{ f(y)+\lambda
_{k}\Gamma \lbrack q(x^{k},y)]\right\}, \qquad k\in \mathbb{N}.
\end{equation}
It takes $\lambda _{k}$ instead of $\lambda _{k+1}$. In this case, the
entrepreneur follows a path of optimal worthwhile changes, $x^{k+1}\in
W_{\lambda _{k}}(x^{k})$, $k\in \mathbb{N}$. In this paper, we will adopt
the Mathematical formulation.

\item [] \textbf{Epsilon} \textbf{inexact proximal algorithms. } Several variants can
be founded in this important literature of what is inexact. Let us consider the version given
in Attouch and Soubeyran~\cite{AtAnt2011} following a long tradition, starting with
Rockafellar~\cite{R76}. In our context, the entrepreneur follows an inexact
proximal algorithm if, each period $k+1$, he can choose a new profile $%
x^{k+1}$\ such that
\[
f(x^{k+1})+\lambda _{k}\Gamma \lbrack
q(x^{k},x^{k+1})]\,\leq f(y)+\lambda _{k}\Gamma \lbrack
q(x^{k},y)]+\varepsilon _{k},\quad y\in X,
\]
given a sequence of nonnegative error terms $\left\{ \varepsilon _{k}\right\}$, i.e,
$
P_{\lambda _{k}}(x^{k},x^{k+1})\leq P_{\lambda _{k}}(x^{k},y)+\varepsilon_{k}$,\ $y\in X$. 
The term $\lambda _{k}$ can be replaced by $\lambda_{k+1}$.

\item [] \textbf{Epsilon inexact proximal algorithms represent a succession of
adaptive satisficing processes. \ }Let $\overline{Q}_{\lambda
_{k}}(x^{k})=\sup \left\{ Q_{\lambda _{k}}(x^{k},y),y\in X\right\} <+\infty $
and $\underline{P}_{\lambda _{k}}(x^{k})=\inf \left\{ P_{\lambda
_{k}}(x^{k},y),y\in X\right\} >-\infty $ be, for each current period $k+1$, the
optimal past values of the ``to be increased" and `to be decreased" proximal
payoffs of this entrepreneur. Let $\overline{Q}_{\lambda
_{k}}(x^{k})-s_{k+1} $ and $\underline{P}_{\lambda _{k}}(x^{k})+s_{k+1}$ be,
in this current period $k+1,$ the current \ satisficing levels of the ``to be
increased" and ``to be decreased" proximal payoffs of the entrepreneur. In
this current period$,$ $s_{k+1}>0$ represents, for the VR approach, a given
satisfacing rate; see \cite{Soubeyran2009,Soubeyran2010}. For an inexact proximal
algorithm, $s_{k+1}=\varepsilon _{k}>0$ is a given error term. Then, in the
context of the VR theory, an inexact proximal algorithm have a new
interpretation. It means that, for each period $k+1$, the new profile $x^{k+1}$
must be satisficing. That is to say, ``to be increased" and ``to be decreased"
proximal payoffs of the entrepreneur must be higher or lower than the current
satisficing level, i.e., $Q_{\lambda _{k}}(x^{k},x^{k+1})\geq \overline{Q}%
_{\lambda _{k}}(x^{k})-\varepsilon _{k}$ for a ``to be increased" \ proximal
payoff and $P_{\lambda _{k}}(x^{k},x^{k+1})\leq \underline{P}_{\lambda
_{k}}(x^{k})+\varepsilon _{k}$ for a ``to be decreased" proximal payoff. Each
period $k+1,$ let us consider the variable satisficing set \textbf{\ }$%
S_{\lambda _{k},\varepsilon _{k}}(x^{k})=\left\{ y\in X,P_{\lambda
_{k}}(x^{k},y)\leq \underline{P}_{\lambda _{k}}(x^{k})+\varepsilon
_{k}\right\}$. Then, an epsilon inexact proximal algorithm is defined by a
succession of repeated decision making problems with changeable spaces and
goals (satisficing levels): find $y\in $\textbf{\ }$S_{\lambda
_{k},\varepsilon _{k}}(x^{k})$, $k\in \mathbb{N}$. They are decision making  problems
with changeable spaces. See Larbani and Yu~\cite{LarbaniYu2012} for different aspects of what
can be change and how (their DMCS approach).
\end{itemize}
\subsection{Marginally Worthwhile Changes}\label{sec:3.2}

Consider the current period $k+1$. Let $\ x=x^{k}\curvearrowright y=x^{k+1}$
be a worthwhile change from $x^{k}$ to $x^{k+1}\in W_{\lambda _{k}}(x^{k})$
and let $x^{k+1}\curvearrowright $ $z\in \mathfrak{M}(x^{k+1})\subset X$ be
a marginal change, where $\mathfrak{M}(x^{k+1})$ is a small neighborhood of $%
x^{k+1}$ in the quasi-metric space $X.$ Then, at each period $k+1,$ the
agent who has done the worthwhile change $y=x^{k+1}\in W_{\lambda
_{k}}(x^{k})$ will stop to prolonge this change if, doing one step more this
period $k+1$, from $x^{k+1}$ to $z\in \mathfrak{M}(x^{k+1}),$ this marginal
change is not worthwhile, i.e., $z\notin W_{\lambda _{k}}(x^{k+1}).$ This is
a generalized stopping rule, a \textquotedblleft not worthwhile marginal
change" condition, that will be used later in the context of proximal
algorithms; see condition \eqref{eq:prox6}.

\subsection{Classification of Inexact Proximal Algorithms: The Separation
Between Weak and Strong Resistance to Change}
\begin{itemize}
\item [] \textbf{Two cases.} The consideration of relative resistance to change
functions $\Gamma[\cdot]$ helps to classify proximal algorithms in two separate
groups. The first case is of strong resistance to change, where $\Gamma[q]=q$ for all $q\geq 0$. This case have been examined in \cite{Soubeyran2009,Soubeyran2010,AtAnt2011}. The second case is of weak resistance to change,
where $\Gamma[q]=q^{2}$ and $q=q(x,y)$ is a distance and
not a quasi distance. This is the traditional case. The literature on this topic
is enormous; see, for example, Moreau \cite{Moreau1965} and Martinet~\cite%
{M70}, as well as in the study of variational inequalities associated to
maximal monotone operators; see Rockafellar~\cite{R76}.

The variational approach which considers relative resistance to change
as a core concept which balances motivation and resistance to change
provides us an extra motivation to develop further the study of proximal
algorithms in a nonconvex and possibly nonsmooth setting where the
perturbation term of the usual proximal point algorithm becomes a
\textquotedblleft curved enough" function of the quasi distance between two
successive iterates. Soubeyran~\cite{Soubeyran2009,Soubeyran2010} and, later, Bento and Soubeyran~\cite{BSPaper1}, in a first paper which paves the way for the present one, have shown
the strong link between a relative resistance to change index with the famous ``loss aversion" index (\cite{Tversky1979,Tversky1991}). The generalized proximal algorithm examined, both, in \cite{BSPaper1} and in the present paper, is new and more adapted for
applications in Behavioral Sciences. Moreover, it retrieves recent
approaches of the proximal method for nonconvex functions; see \cite%
{Attouch2009, Moreno2011}. 

\item [] \textbf{Hypothesis on the relative resistance to change. }In the remainder
of this paper we assume that $\Gamma $ is a
twice differentiable function such that: 
\begin{equation}
\Gamma \lbrack 0]=\Gamma ^{\prime }[0]=0,\quad \mbox{and}\quad \Gamma
^{\prime }[q]>0,\quad \Gamma ^{\prime \prime }[q]>0,\quad q>0,
\label{condDC1}
\end{equation}%
and there exist constants $r,\bar{q},\bar{\rho}_{\Gamma }(r)>0$, satisfying
the following condition: 
\begin{equation}
\Gamma ^{\prime }[q/r]\leq \bar{\rho}_{\Gamma }(r)\Gamma \lbrack q]/q,\qquad
0<q\leq \bar{q}.  \label{NewcondDC2}
\end{equation}%
Let us consider a generalized rate of curvature of $\Gamma $ given by: 
\begin{equation}
\rho _{\Gamma }(q,r):=\frac{\Gamma ^{\prime }[q/r]}{\left( \Gamma \lbrack
q]/q\right) },\qquad 0<q\leq \bar{q}.  \label{ratecurv100}
\end{equation}%
In the particular case $r=1$, \eqref{ratecurv100} represents, in Economics,
the elasticity of the disutility curve $\Gamma $; see, for instance, \cite%
{Soubeyran2009, Soubeyran2010}. From \eqref{ratecurv100}, condition %
\eqref{NewcondDC2} is equivalent to the condition: 
\begin{equation*}
\bar{\rho}_{\Gamma }(r)=\sup \{\rho _{\Gamma }(q,r):0<q<\bar{q}\}<+\infty
,\qquad r\in ]0,1[\quad \mbox{fixed}.
\end{equation*}%
Let us consider, for each $\alpha >1$ fixed, the function $\Gamma \lbrack
q]:=q^{\alpha }$. It is easy to see that, in this case, $\bar{\rho}%
_{D}(r)\in \lbrack \alpha r^{1-\alpha },+\infty )$. In particular, we can
take 
\begin{equation}
\bar{\rho}_{\Gamma }(q,r)=\alpha r^{1-\alpha }=\bar{\rho}_{\Gamma
}(r)<+\infty .  \label{curvature10}
\end{equation}%
More accurately, for each $\alpha >1$, $\Gamma \lbrack q]=q^{\alpha }$
represents a disutility of costs to change. It is strictly increasing and
satisfies \eqref{condDC1} and \eqref{NewcondDC2}.
\end{itemize}
\section{An Inexact Proximal Point Algorithm: Convergence to a Weak or
Strong Variational Trap}

\label{sec4}

\subsection{End Points as Critical Points or Variational Traps}

In a first paper, Bento and Soubeyran~\cite{BSPaper1} showed when, in a quasi metric
space, a generalized inexact proximal algorithm, equipped with a generalized
perturbation term $\Gamma \left[ q(x,y)\right] $, and defined, each step by,
i) a sufficient descent condition and, ii) a stopping rule, converges to a
critical point. Then, they have shown that the speed of convergence and
convergence in finite time depends of the curvature of the perturbation term
and of the Kurdyka-Lojasiewicz property associated to the objective
function. A striking and new application has been given. It concerns the
impact of the famous ``loss aversion effect" (Nobel Prize \cite{Tversky1979,Tversky1991}) on the speed of convergence of the
generalized inexact proximal algorithm. However, in the context of the \textquotedblleft Variational rationality
approach", which considers, as central dynamical concepts, worthwhile stay
and change processes, these important results in Applied Mathematics are not
enough, from the viewpoint of our applications to Behavioral Sciences,
unless we can show that this critical point is a variational trap (strong or
weak) where the agent will prefer to stay than to move, because his
motivation to change is strictly or weakly lower than his resistance to
change. This section presents, under the conditions of \cite[Theorem 3.1]%
{BSPaper1}, a worthwhile stay and change process which converges to a
critical point of $f$ which is a weak trap (compare, below, with the
definition of a strong global trap). Then, start with the general definition
of a weak global trap instead of a strong one (\cite{Soubeyran2009,Soubeyran2010}).

\begin{defn}
\label{def:trap} Let $x\in X$ be a given action and $\xi >0$ be a
satisficing rate of change choosen by the agent. Let $W_{\xi }(x):=\left\{
y\in X,M(x,y)\geq \xi R(x,y)\right\} $ be his worthwhile to change set,
starting from $x\in X$. Then, starting from $x^{\ast }\in X$ with a given
satisficing worthwhile to change rate $\xi ^{\ast }>0,$ a strong variational
trap $x^{\ast }\in X$ is such that motivation to change is stricly lower
than resistance to change, $M(x^{\ast },y)<\xi _{\ast }R(x^{\ast },y)$ for
all $y\neq x^{\ast }\in X.$ A weak variational trap is such that $M(x^{\ast
},y)\leq \xi _{\ast }R(x^{\ast },y)$, for all $y\in X.$ This defines the stationary side of a trap. The variational aspect comes from being the end of a worthwhile to change process, starting from an initial given point.
\end{defn}

\begin{rmk}$\;$ \label{remark:trap2}
\begin{itemize}
\item [a)] Notice that a strong global trap is such that $W_{\xi _{\ast}}(x^{\ast })=\left\{ x^{\ast }\right\} $ and a weak global trap is such
that $W_{\xi _{\ast }}(x^{\ast })=\left\{ y\in X,\text{ }M(x^{\ast },y)=\xi_{\ast }R(x^{\ast },y)\right\} $. At a strong (weak) global trap, the agent
strictly (weakly) prefers to stay than to move. Then, when a process of
worthwhile stays and changes converges to a strong variational trap, this
variational formulation defines, starting from an initial point, a
variational trap as the end point of a path of worthwhile changes,
worthwhile to approach, but not worthwhile to leave. This because, starting
from there, there is no way to do any other worthwhile change, except
repetitions.

\item [b)] Assuming that $\{\lambda _{k}\}$ converges to $\lambda _{\infty }$, our
sufficient condition proposes an algorithm which, following a succession of
worthwhile changes $x^{k+1}\in W_{\lambda _{k}}(x^{k}),k\in \mathbb{N}$,
converges to a weak global trap $x^{\ast }$ such that $W_{\lambda _{\infty
}}(x^{\ast })=\left\{ y\in X:\;M(x^{\ast },y)=\lambda _{\infty }R(x^{\ast
},y)\right\} $. Since the agent is free to choose all his satisficing
worthwhile to \ change rates $\lambda _{k}$ in an adaptive way, this will
show that the agent, choosing at the limit point $x^{\ast }$ a satisicing
worthwhile to change rate $\lambda _{\ast }>\lambda _{\infty },$ ends in a
strong global trap $x^{\ast }$, because $M(x^{\ast },y)=\lambda _{\infty
}R(x^{\ast },y)<\lambda _{\ast }R(x^{\ast },y)$, for all $y\in X$.

\item [c)] As observed in Section~\ref{sec2}, in the specific context of this
paper, we have 
\begin{equation*}
M(x,y)=U\left[ A(x,y)\right] =f(x)-f(y)
\end{equation*}%
\begin{equation*}
R(x,y)=D\left[ C(x,y)\right] ,\quad \Gamma \left[ q(x,y)\right] =\lambda
U^{-1}\left[ D[q(x,y)]\right] ,\xi =1.
\end{equation*}
\end{itemize}
Then, in our present paper, a strong (resp. weak) variational trap is
such that $f(x^{\ast })-f(y)<\lambda\Gamma \left[ q(x^{\ast },y)\right]$, for all $y\neq x^{\ast }$ (resp. $f(x^{\ast })-f(y)\leq \lambda \Gamma \left[ q(x^{\ast },y)\right]$, for all $y\in X$).

\end{rmk}

\subsection{Some Definitions from Subdifferential Calculus}

In this section some elements concerning the subdifferential calculus are
recalled; see, for instance, \cite{Rockafellar1998, Mordukhovich2006}. Assume that $f:\mathbb{R}^{n}\to\mathbb{R}\cup\{+\infty\}$ is a
proper lower semicontinuous function. The domain of $f$, which we denote by %
\mbox{dom}$f$, is the subset of $\mathbb{R}^{n}$ on which $f$ is
finite-valued. Since $f$ is proper, then \mbox{dom}$f\neq \emptyset$.

\begin{defn}
\label{defi:subFL}$\;$

\begin{itemize}
\item[i)] The Fr\'{e}chet subdifferential of $f$ at $x\in \mathbb{R}^{n}$,
denoted by $\hat{\partial}f(x)$, is the set given by: 
\begin{equation*}
\hat{\partial}f(x):=\left\{%
\begin{array}{c}
\{x^{\ast}\in\mathbb{R}^n: \displaystyle\liminf_{y\to x; y\neq x}\frac{1}{%
\|x-y\|}(f(y)-f(x)-\langle x^{\ast},y-x \rangle)\geq 0\}, \; if\; x\in%
\mbox{dom} f, \\ 
\emptyset,\; \hspace{7.9cm} if \; x\notin\mbox{dom} f.%
\end{array}%
\right.
\end{equation*}

\item[ii)] The limiting Fr\'{e}chet subdifferential (or simply
subdifferential) of $f$ at $x\in\mathbb{R}^n$, denoted by $\partial f(x)$,
is the set given by: 
\begin{equation*}
\partial f(x):=\left\{ 
\begin{array}{c}
\{x^{\ast}\in\mathbb{R}^n|\exists x_{n}\rightarrow x, f(x_n)\rightarrow
f(x),x_{n}^{\ast}\in\hat{\partial}f(x_n);x_{n}^{\ast}\rightarrow
x^{\ast}\},\; if \; x\in \mbox{dom}f. \\ 
\emptyset,\; \hspace{7.6cm} if\; x\notin\mbox{dom} f.%
\end{array}%
\right.
\end{equation*}
\end{itemize}
\end{defn}

Throughout the paper we consider the subdifferential $\partial f$ since it
satisfies a closedness property important in our convergence analysis, as
well as in any limiting processes used in an algorithmic context.

A necessary condition for a given point $x\in \mathbb{R}^{n}$ to be a
minimizer of $f$ is 
\begin{equation}  \label{inclusionCritc}
0\in \partial f(x).
\end{equation}
It is known that, unless $f$ is convex, \eqref{inclusionCritc} is not a
sufficient condition. The domain of $\partial f$, which we denote by %
\mbox{dom} $\partial f$, is the subset of $\mathbb{R}^{n}$ on which $%
\partial f$ is a nonempty set. In the remainder, a point that satisfies %
\eqref{inclusionCritc} is called limiting-critical or simply critical point.


\subsection{The Algorithm}


In \cite{AtAnt2011} the authors examined the ``local epsilon inexact
proximal" algorithm, 
\[
f(x^{k+1})+\lambda _{k}d(x^{k},x^{k+1})\,\leq
f(y)+\lambda _{k}d(x^{k},y)+\varepsilon _{k},\qquad y\in E(x^{k},r_{k})\subset X,
\]
where, i) $d$ is a distance, ii) $E(x^{k},r_{k+1})\subset X$ \
is a variable choice set (a moving ball), for each current period $k+1$. Following \cite{AtAnt2011} we consider the so called
\textquotedblleft global epsilon inexact proximal" algorithm as follows: starting from the
current position $x^{k}$, let us define the next iterate $x^{k+1}$ as
follows:
\begin{equation}\label{newinexctprox2}
f(x^{k+1})+\lambda _{k}\Gamma \lbrack q(x^{k},x^{k+1})]\,\leq f(y)+\lambda
_{k}\Gamma \lbrack q(x^{k},y)]+\varepsilon _{k},\qquad y\in X,
\end{equation}
where $\left\{ \lambda _{k}\right\} ,\left\{ \varepsilon _{k}\right\} $ are
given sequences of nonnegative real numbers, and $q$ is a quasi distance.
In the particular case where the generalized perturbation term $\Gamma \lbrack q(x,y)]=q(x,y)^{2}$ and 
$q(x,y)=d(x,y)$ is a distance, instead of a quasi distance, our ``global epsilon
inexact proximal" algorithm coincides with the case considered by Zaslavski~\cite{Zaslavski2011}.


\begin{assumption}\label{Assumption1} There exist $\beta _{1},\beta _{2}\in \mathbb{R}%
_{++}$ such that: $\beta _{1}\Vert x-y\Vert \leq q(x,y)\leq \beta _{2}\Vert
x-y\Vert $, $x,y\in \mathbb{R}^{n}$.
\end{assumption}

\noindent This is the case in our knowledge management example. For an
other explicit example where inconvenients to change are a quasi-distance
satisfying Assumption~\ref{Assumption1}, see \cite{Moreno2011}.
Next, we recall the inexact version of the proximal point method introduce
in \cite{BSPaper1}.

\begin{algorithm}
\label{algor1} Take $x^{0}\in \mbox{dom}f$, $0<\bar{\lambda}\leq \tilde{%
\lambda}<+\infty $, $\sigma \in \lbrack 0,1[$ and $b>0$. For each $%
k=0,1,\ldots $, choose $\lambda _{k}\in \lbrack \bar{\lambda},\tilde{\lambda}%
]$ and find $(x^{k+1},w^{k+1},v^{k+1})\in \mathbb{R}^{n}\times \mathbb{R}%
^{n}\times \mathbb{R}^{n}$ such that: 
\begin{equation}
f(x^{k})-f(x^{k+1})\geq {\lambda _{k}}(1-\sigma )\Gamma \lbrack
q(x^{k},x^{k+1})],  \label{eq:prox4}
\end{equation}%
\begin{equation}
w^{k+1}\in \partial f(x^{k+1}),\quad v^{k+1}\in \partial q(x^{k},\cdot
)(x^{k+1}),
\end{equation}%
\begin{equation}
\Vert w^{k+1}\Vert \leq b\Gamma ^{\prime }[q(x^{k},x^{k+1})]\Vert
v^{k+1}\Vert ,  \label{eq:prox6}
\end{equation}

The first condition is a sufficient descent condition. It is \ a
(proximal-like) worthwhile to change condition $x^{k+1}\in W_{\xi
_{k+1}}(x_{k})$, where the proximal perturbation term \ defines the
relative resistance to change function. This condition tells
us that it is worthwhile to change  from $x^{k}$ to $x^{k+1}$, rather than to stay at $x^{k}$. In this case,
advantages to change from $x^{k}$ to $x^{k+1},$ $A(x^{k},x^{k+1})=f(x^{k})-f(x^{k+1})$ are, each period, higher than
some adaptive proportion $\xi _{k+1}=\lambda _{k}(1-\sigma )$ of
the relative disutility of inconvenients to change rather than to stay $\Gamma \lbrack q(x^{k},x^{k+1})]=U^{-1}\left[ D\left[ I(x^{k},x^{k+1})\right]\right]$, where, i) inconvenients to change \ rather than to stay are $I(x^{k},x^{k+1})=C(x^{k},x^{k+1})-C(x^{k},x^{k})=q(x^{k},x^{k+1})$, ii) costs to be able to change from $x^{k}$ to $x^{k+1}$
are $C(x^{k},x^{k+1})=q(x^{k},x^{k+1}),$ while, iii) costs to be able to stay $C(x^{k},x^{k})=q(x^{k},x^{k})=0$ are
zero as quasi distances. The second conditions defines subgradients of the
objective and costs to be able to change functions. The third condition \ is
a stopping rule which says, each period, when the agent prefers \ to do not
make a \ new marginal change, because it is not worthwhile to do it, this
period; see Section \ref{sec:3.2} on marginally worthwhile changes.
\end{algorithm}

\begin{rmk}
As pointed out by the authors, Algorithm~\ref{algor1} retrieves the inexact
algorithm proposed in \cite[Algorithm 2]{Attouch2010} in the particular case 
$\Gamma \lbrack q]=q^{2}/2$, $q(x,y)=\Vert x-y\Vert $ and $1-\sigma =\theta $%
. Moreover, Algorithm~\ref{algor1} is an habituation/routinization process
and any sequence generated from it is a path of worthwhile changes with
parameter $\lambda _{k}(1-\sigma )$ such that, at each step, it is
marginally worthwhile to stop. The variational stopping rule condition
raises the following question: when, marginally, a change stops to be
worthwhile? This strongly depends on the shapes of the utility and
desutility functions.
\end{rmk}


Comparing Algorithm~\ref{algor1} with the iterative process %
\eqref{newinexctprox2}, we observe the following:

\begin{itemize}
\item[i)] on one side, the iterative process \eqref{newinexctprox2} is much
more specific than our Algorithm~\ref{algor1}. Indeed the weak ``worthwhile
to change" condition \eqref{eq:prox4} is replaced by the much stronger
condition \eqref{newinexctprox2}.


\item[ii)] on the other side, the iterative process \eqref{newinexctprox2}
does not impose the \textquotedblleft not worthwhile marginal change
condition" \eqref{eq:prox6} as the Algorithm~\ref{algor1} does.
\end{itemize}

Next we propose a new inexact proximal algorithm, combining a particular instance
of \eqref{newinexctprox2} with the stopping rule \eqref{eq:prox6}.


\begin{algorithm}
\label{algor2} Take $x^{0}\in \mbox{dom}f$, $0<\bar{\lambda}\leq \tilde{%
\lambda}<+\infty $, $\sigma \in \lbrack 0,1[$ and $b>0$. For each $%
k=0,1,\ldots $, choose $\lambda _{k}\in \lbrack \bar{\lambda},\tilde{\lambda}%
]$ and find $(x^{k+1},w^{k+1},v^{k+1})\in \mathbb{R}^{n}\times \mathbb{R}%
^{n}\times \mathbb{R}^{n}$ such that: 
\begin{equation}
f(y)-f(x^{k+1})\geq {\lambda _{k}}\left[(1-\sigma )\Gamma \lbrack
q(x^{k},x^{k+1})]-\Gamma \lbrack q(x^{k},y)]\right],\qquad y\in X,
\label{eq:prox1000}
\end{equation}%
\begin{equation}
w^{k+1}\in \partial f(x^{k+1}),\quad v^{k+1}\in \partial q(x^{k},\cdot
)(x^{k+1}),
\end{equation}%
\begin{equation}
\Vert w^{k+1}\Vert \leq b\Gamma ^{\prime }[q(x^{k},x^{k+1})]\Vert
v^{k+1}\Vert .  \label{eq:prox1001}
\end{equation}
\end{algorithm}

\begin{rmk}
This new inexact proximal algorithm imposes a stronger worthwhile to
change condition than Algorithm~\ref{algor1}, because it must be
verified, each period, for each $y\in X$. Setting $y=x^{k}$ gives the last worthwhile to change condition. The other two
conditions remain unchanged. Note that the exact proximal algorithm %
\eqref{eq:prox1-2} is a specific case of our new algorithm (it holds by
taking $\sigma =0$).
%
%
%
%
%
%
The new inexact worthwhile to change condition is $\ P_{\lambda
_{k}}(x^{k},x^{k+1})\leq P_{\lambda _{k}}(x^{k},y)+\lambda _{k}\sigma \Gamma %
\left[ {q(x^{k},x^{k+1})}\right]$, for all $y\in X$.
\end{rmk}

As in \cite{Attouch2009,Moreno2011,Attouch2010,BSPaper1}, our main
convergence result is restricted to functions that satisfy the so-called
Kurdyka-Lojasiewicz inequality; see, for instance, \cite{Lojasiewicz1963,Kurdyka1998,Bolte2006,Attouch2010-2}.
Next formal definition
of the Kurdyka-Lojasiewicz inequality can be finding in \cite{Attouch2010-2}, where it is also possible to find several examples and a good discussion
over important classes of functions which satisfy the mentioned inequality.

\begin{defn}
A proper lower semicontinuous function $f:\mathbb{R}^{n}\to \mathbb{R} \cup
\{+\infty\}$ is said to have the Kurdyka-Lojasiewicz property at $\bar{x}\in%
\mbox{dom}\; \partial f$ if there exists $\eta \in ]0,+\infty]$, a
neighborhood $U$ of $\bar{x}$ and a continuous concave function $%
\varphi:[0,\eta[\rightarrow \mathbb{R}_+$ such that: 
\begin{equation}  \label{eq:kur101}
\varphi(0)=0,\quad \varphi\in C^1(0,\eta), \quad \varphi ^{\prime
}(s)>0,\quad s\in]0,\eta[;
\end{equation}
\begin{equation}  \label{eq:kur100}
\varphi ^{\prime }(f(x)-f(\bar{x}))dist(0, \partial f(x))\geq 1,\quad x\in
U\cap [f(\bar{x})<f<f(\bar{x})+\eta],
\end{equation}

\begin{itemize}
\item \textit{$dist(0,\partial f(x)):=inf\{\| v \|: v \in \partial f(x)\}$, }

\item \textit{$[\eta_1 <f<\eta_2]:=\{x\in M: \eta_1 < f(x) < \eta_2\},\quad
\eta_1<\eta_2$. }
\end{itemize}
\end{defn}

In what follows, we assume that  $f$ a is bounded from below, continuous on \mbox{dom}$f$ and KL function, i.e., a
function which satisfies the Kurdyka-Lojasiewicz inequality at each point of $\mbox{dom}\partial f$.

\begin{thm}
Assume that $\{x^{k}\}$ is bounded sequence generated from Algorithm~\ref{algor2}, $\tilde{x}$ is an accumulation point of $\{x^{k}\}$ and Assumption~\ref{Assumption1} holds. Let $U\subset\mathbb{R}^{n}$ be a neighborhood of $%
\tilde{x}$, $\eta\in ]0,+\infty]$ and $\varphi:[0,\eta[\to\mathbb{R}_{+}$ a
continuous concave function such that \eqref{eq:kur101} and \eqref{eq:kur100}
hold. If $\delta\in (0,\bar{q})$ $($see condition \eqref{NewcondDC2}$)$ and $%
r\in ]0,1[$ are fixed constants, $B(\tilde{x}, \delta/\beta_{1})\subset U$, $%
a:=\bar{\lambda}(1-\sigma)$ and $M:=\frac{Lb}{a}$, then  the whole sequence $\{x^{k}\}$ converges to a critical point $x^{\ast}$ of $f$
which is a strong global trap, relative to the worthwile to
change set $W_{\lambda _{\ast }}(x^{\ast })$, for any choice of the final
satisficing rate $\lambda _{\ast }>\lambda _{\infty }$.
\end{thm}
\begin{proof}
The first part of the theorem follows immediately from \cite[Theorem 3.1]{BSPaper1} because any sequence, generated from Algorithm~\ref{algor2}, satisfies the conditions \eqref{eq:prox4} and \eqref{eq:prox6} of Algorithm~\ref{algor1}. Let $x^{\ast}$ be the limit point of the sequence $\{x^{k}\}$. Given that the sequence $\{\lambda_{k}\}\subset [\bar{\lambda},\tilde{\lambda}]$ (is bounded), $0<\bar{\lambda}\leq \tilde{\lambda}<+\infty$,  taking a subsequence, if necessary, we can assume that  $\lambda_{k}$ converges  to a certain $\lambda_{\infty}\in ]0,+\infty[$. For the second part, note that  $\{f(x^{k})\}$ is a non increasing sequence and $x^{\ast}\in\mbox{dom}$$f$. Now, given that $q(\cdot, y)$ is continuous for each $y\in X$ (see \cite{Moreno2011}), $\Gamma$ is continuous and $f$ is continuous on $\mbox{dom} f$, taking the limit in  \eqref{eq:prox1000}  as $k$ goes to infinity and assuming that $\lambda_{k}$ converges  to a certain $\lambda_{\infty}\in ]0,+\infty[$, we get:
\[
f(x^{\ast })\leq  f(y)+\lambda_{\ast }\Gamma[q(x^{\ast },y)], \qquad y\in X.
\]
Therefore, the desired result follows from Remark~\ref{remark:trap2}.
\end{proof}

%

\section{Application to the Formation and Break of Habits/Routines}

To save space, this section represents a really very short summary of the ``available to read? long preprint of our present paper, named ``Some comparisons between the Variational rationality, Habitual domain and DMCS approaches"; see  \cite{BentoAntoine2014-2}. The {\it Variational rationality (VR) approach} (see \cite{Soubeyran2009,Soubeyran2010})
focus attention on interdisciplinary stability and change dynamics (habits
and routines, creation and innovation, exploration and exploitation,\ldots),
and the self regulation problem, seen as a stop and go course pursuit
between feasible means and desirable ends mixing, in alternation, discrepancy
production (goal setting, goal revision), discrepancy reduction (goal
striving, goal pursuit) and goal disengagement. It rests on two main
concepts, worthwhile temporary stays and changes, variational traps and nine
principles. This (VR) approach allows to recover the main mathematical
variational principles, and in turn, it benefits from almost all variational
algorithms for procedural applications, which all, use some of the main
variational rationality principles.
%
The {\it Habitual domain (HD) theory and (DMCS)
approach} (see Yu and Chen~\cite{YuChen2010}, for a nice presentation) and the (DMOCS) Decision
making and optimization problems in changeable spaces (see Larbani and Yu~\cite{LarbaniYu2012}) refer to three stability and change problems. They are i) stability
issues, using a system of differential equations, a variant of the famous
pattern formation Cohen-Grossberg model (see Cohen and Grossberg~\cite{Cohen1983}), ii)
expansion of an initial competency set to be able to solve a given problem,
which requires to acquire a new given competency set, using mathematical
programming methods and graphs, iii) DMCS optimization and game problems,
using Markov chains, with applications to innovation cover-discover problems
(see Yu and Larbani~\cite{YuLarbani2009} and Larbani and Yu~\cite{LarbaniYu2009,LarbaniYu2011}).

\subsection*{\textbf{1) Habit/routine formation and break problems}}

The (VR) approach sees habit/routine formation and break as a balance between motivation
and resistance to change, when agents use worthwhile changes; Permanent
habits refer to variational traps, as the end of a succession of worthwhile
changes. These findings fit well with Psychology and Management theories of
habits and routines formation and break. For example, in Psychology, habits
form by repetitions, in a rather stable context, which trigger their
repetition. They become gradually more and more automatized behaviors which
are intentional (goal directed), more or less conscious and controllable,
and economize cognitive resources. In this context agents are bounded
rational. The (HD) approach modelizes habit formation as a balance, each
step, between excitation and inhibition forces, which determine, each time,
the variation of the allocation of attention and effort (propensities),
using, as said before, a variant of the Cohen-Grossberg system of
differential equations (see \cite{Cohen1983}).

In the limit, the allocation of attentions and efforts (propensities)
converge to an allocation which represents a stable habitual domain. The
(DMOCS) approach of games defines limit profiles of mind sets as absorbing
states of a Markow chain. In these two contexts agents are bounded rational.

\subsection*{\textbf{2) Inexact proximal algorithms as repeated
satisficing problems with changeable decision sets}}

We have shown in Section 3 of this paper that
our inexact proximal algorithms refer to variable and changeable decision
sets, payoffs, goals (satisficing worthwhile changes) and preference
processes. This is the case because worthwhile to change sets are changeable decision sets which change, each period, with experience and the choice, each period, of the satisficing worthwhile to change ratio; see the whole Section 2 and the dynamic of worthwhile (hence satisficing) temporary stays and changes. More precisely Inexact generalized proximal algorithms are
specific instance of VR worthwhile stay and change adaptive dynamics. They
consider non transitive variable worthwhile to change preferences, see Section 2, which are
reference dependent preferences with variable reference points (variable
experience dependent utility/disutility functions, variable payoff
functions, variable and non linear resistance to change functions  via, in each case, of the introduction of the separable and variable lambda term). They can
use costs to be able to change which do not satisfy the triangular
inequality (\cite{BentoNetoAntoine2014} and several other references within \cite{BentoAntoine2014-2}). They are inexact and
procedural algorithms on quasi metric spaces. Then, they modelize bounded
rational and more or less myopic agents who bracket difficult decisions in
several steps, contrary to exact proximal algorithms which modelize a
repeated optimization problem, and which are not the topic of our paper.
They generalize the Simon satisficing principle to a dynamical context
(repeated and adaptive satisficing). They are anchored to optimization
processes as benchmark cases. They deal with changeable spaces as changeable worthwhile to change sets and also changeable exploration sets (see
\cite{AtAnt2011,BentoNetoAntoine2014}) and consider
convergence in finite time as a central topic (see \cite{BSPaper1,BCSS2013}). They include psychological aspects (like motivations, cognitions
and inertia) and can easily include emotional aspects.

We point out that the assumptions of inexact proximal
algorithms explicit in behavioral terms have been given in Section 2 of this paper (see the example) and after
each proximal algorithm.

\section{Conclusion}

In this paper, following \cite{BSPaper1}, and using
the recent variational approach presented in \cite{Soubeyran2009,Soubeyran2010}, we have proposed a
generalized ``epsilon inexact proximal" algorithm that converges to a
critical point which is also a variational trap. In Mathematics, our paper
helps to show how the literature on proximal algorithms can be divided in
two parts: the case of strong and weak relative resistance to change. In
this paper we have considered the most difficult situation, the weak case.
In Behavioral Sciences, our paper offers a dynamic model for
habituation/routinization processes, and gives a striking and new result on
the impact of the famous ``loss aversion" index (see the Nobel Prize
\cite{Tversky1979,Tversky1991}) on the speed of
convergence of such processes. Given editorial constraints (lack of space in
the present paper), this important result appears in the first paper (\cite{BSPaper1}). In \cite{BentoAntoine2014-2} the authors compare our VR variational rationality
approach of inexact proximal algorithms to the HD habitual domain theory and
DMCS approach (\cite{Yu1990,LarbaniYu2012}). Future research will consider
the multiobjective case.

\end{document}